\documentclass[preprint,12pt]{amsart}

\usepackage[margin=1.5in]{geometry}

\def\Z{{\mathbb{Z}}}

\def\C{{\mathbb{C}}}
\usepackage{amsmath}
\usepackage{amsfonts}
\usepackage{amssymb}
\usepackage{xypic}
\usepackage{longtable}
\usepackage{amsthm}
\usepackage{enumerate}
\usepackage{graphicx}

\theoremstyle{plain}
\newtheorem{theorem}{Theorem}[section]
\newtheorem{corollary}[theorem]{Corollary}
\newtheorem{lemma}[theorem]{Lemma}
\newtheorem{proposition}[theorem]{Proposition}
\theoremstyle{definition}

\newtheorem{example}[theorem]{Example}
\newtheorem{rem}[theorem]{Remark}

\makeatletter
\def\ps@pprintTitle{%
  \let\@oddhead\@empty
  \let\@evenhead\@empty
  \let\@oddfoot\@empty
  \let\@evenfoot\@oddfoot
}
\makeatother

\title{Fixed points of endomorphisms of complex tori}
\author{Mat\'{\i}as Alvarado and Robert Auffarth}
\address{M. Alvarado \\Departamento de Matem\'aticas, Facultad de
Ciencias, Universidad de Chile, Santiago\\Chile}
\email{matiass\_lh@hotmail.com}
\address{R. Auffarth \\Departamento de Matem\'aticas, Facultad de
Ciencias, Universidad de Chile, Santiago\\Chile}
\email{rfauffar@uchile.cl}

\thanks{The first author was partially supported by a CONICYT masters grant and by Fondecyt Grant 1150943. The second author was partially supported by Fondecyt Grant 3150171. Both authors were additionally partially supported by CONICYT PIA ACT1415}


\begin{document}

\maketitle

\begin{abstract}
We study the asymptotic behavior of the cardinality of the fixed point set of iterates of an endomorphism of a complex torus. We show that there are precisely three types of behavior of this function: it is either an exponentially growing function, a periodic function, or a product of both. \end{abstract}

\section{Introduction}

Let $X$ be a complex manifold and let $f:X\to X$ be a holomorphic map. It is a natural question to ask how many fixed points $f$ has, and what geometric and topological information they can give us about $X$. The answer to this question of course depends on $X$ and the map chosen, and as stated the question is too broad to be useful. As noted in the introduction of \cite{BH}, one expects a more uniform answer by looking at asymptotic properties of fixed points of iterates of $f$, as this point of view has been succesfully adopted in several areas of complex and algebraic geometry (see \cite{BH} and the references cited in the introduction). Let 
$$F_f(n)=F(n):=\#\mbox{Fix}(f^n)$$ 
denote the number of fixed points of $f^n$ if this number is finite, and 0 if not. 

We are interested in studying the asymptotic behavior of $F(n)$ in the case that $X$ is a complex torus. In this case, the Lefschetz Fixed-Point Theorem for complex tori gives the exact number of fixed points of $f$ in terms of the eigenvalues of the action of $f$ on the tangent space at the origin. Indeed, let $\rho_a(f)\in \mbox{End}(T_X(0))$ denote the differential of $f$ at 0 (i.e. the analytic representation of $f$), and let $\lambda_1,\ldots,\lambda_g$ be its eigenvalues. Then by \cite[13.1.2]{BL}, 
\begin{equation}\label{eigen}\#\mbox{Fix}(f^n)=\left|\prod_{i=1}^g(1-\lambda_i^n)\right|^2.\end{equation}

This shows that the asymptotic behavior of $F(n)$ is governed by the eigenvalues of $\rho_a(f)$. This is not surprising, given the fact that the eigenvalues of $\rho_a(f)$ control many of the topological properties of the dynamical system induced by $f$ (such as its topological entropy; see the next section for details).

Our motivation for studying this problem for complex tori comes from \cite{BH}, where the authors give a complete classification of the behavior of $F(n)$ in the case that $X$ is a two-dimensional complex torus. Indeed, it is shown that on a two-dimensional complex torus, $F(n)$ either 

\begin{itemize}
\item grows exponentially
\item is periodic 
\item is a product of these two behaviors.
\end{itemize}

One of the key lemmas that is used in this classification is the fact that if $\lambda$ is an eigenvalue of an endomorphism of a complex two-dimensional torus and $|\lambda|=1$, then $\lambda$ must be a root of unity. This is no longer true for higher dimensions, and this, a priori, produces one of the main difficulties for extending the results found in dimension two.

It seems that the appearance of eigenvalues that lie on the unit circle and are not roots of unity should be interesting in their own right. For example, in \cite{OT}, Oguiso studies Salem numbers that appear as eigenvalues of endomorphisms of complex tori in the context of studying the existence of equivariant holomorphic fibrations between tori. For this reason, we briefly study properties of endomorphisms with these types of eigenvalues, as well as find examples of when these eigenvalues appear. 

Although interesting eigenvalues appear for higher dimensional complex tori, our main theorem shows that the asymptotic behavior of $F(n)$ is virtually the same as in dimension two. Indeed, we prove:

\begin{theorem}\label{main}
Let $X$ be a complex torus of dimension $g$ and let $f$ be an endomorphism of $X$. Then $\#\mbox{Fix}(f^n)$ has one of the following behaviors:
\begin{enumerate}
\item It grows exponentially in $n$, which means that there are real constants $A,B>1$ and an integer $N$ such that for all $n\geq N$, $A^n\leq\#\mbox{Fix}(f^n)\leq B^n$.
\item It is a periodic function, and the non-zero eigenvalues of $f$ are $k$-th roots of unity where $k$ is contained in the set $\{n\in\mathbb{N}:\varphi(n)\leq 2g\}$ where $\varphi$ is Euler's function.
\item There exist integers $n_1,\ldots,n_r\geq2$ and an exponentially growing function $h:\mathbb{N}\to\mathbb{N}$ such that
$$\#\mbox{Fix}(f^n)=\left\{\begin{array}{ll}0&\mbox{if }n\equiv 0\mbox{ (mod }n_i)\mbox{ for some }i\\h(n)&\mbox{otherwise}\end{array}\right.$$
\end{enumerate} 

Moreover, types (1) and (2) appear for simple abelian varieties, but type (3) never appears on a simple torus.
\end{theorem}

Our proof goes along a different line than Bauer and Herrig's proof for dimension two. Indeed, if $f$ is an endomorphism of a complex torus $X$ and $\chi_f^r(t)$ is the characteristic polynomial of its action on $H_1(X,\mathbb{Z})$, we use the logarithmic Mahler measure of $\chi_f^r(t)$ from number theory to study $F(n)$. From a dynamical point of view, the logarithmic Mahler measure is just the topological entropy of $f$.  By using Baker's Theorem on the independence of logarithms of algebraic numbers, along with the logarithmic Mahler measure, we are able to deal with endomorphisms whose analytic representation has an eigenvalue that lies on the unit circle but is not a root of unity. 

As a corollary, we obtain:

\begin{corollary}
An abelian variety $X$ is simple if and only if for every $f\in\mbox{End}(X)$, $\#\mbox{Fix}(f^n)$ is either periodic or has exponential growth.
\end{corollary}

\noindent\textit{Acknowledgements:} We would like to thank Eduardo Friedman for many beneficial conversations.

\section{Eigenvalues of endomorphisms of complex tori}

Let $X=\mathbb{C}^g/\Lambda$ be a complex torus of dimension $g$ where $\Lambda\subseteq\mathbb{C}^g$ is a full-rank lattice, and let $f$ be an endomorphism of $X$. We will assume that $f$ fixes the origin, since by \cite{BL2} $\#\mbox{Fix}(f)=\#\mbox{Fix}(f-f(0))$. Moreover, as is customary in the area of abelian varieties, $\mbox{End}(X)$ will denote the algebra of holomorphic maps from $X$ to itself that fix the origin.  We have the analytic and rational representations 
$$\rho_a:\mbox{End}(X)\to\mbox{End}(T_X(0))$$ 
$$\rho_r:\mbox{End}(X)\to\mbox{End}(H_1(X,\mathbb{Q})),$$
that are the representations of $f$ on the tangent space of $X$ at 0 and on the lattice $\Lambda\simeq H_1(X,\mathbb{Z})$, respectively. We will oftentimes identify these representations with concrete matrix representations if we have some fixed basis in mind. If $\lambda_1,\ldots,\lambda_g$ are the eigenvalues of $\rho_a(f)$, then the Holomorphic Lefschetz Fixed-Point Formula (\ref{eigen}) shows that the number of fixed points of an iteration of $f$ is governed by the powers of these. By abuse of terminology we will say that these eigenvalues are eigenvalues of $f$. Each of these representations comes with a characteristic polynomial:
$$\chi_f^a(t):=\det(tI_g-\rho_a(f))$$
$$\chi_f^r(t):=\det(tI_{2g}-\rho_r(f)).$$
It is well known that $\chi_f^r(t)=\chi_f^a(t)\overline{\chi_f^a(t)}$ since $\rho_r\simeq\rho_a\oplus\overline{\rho_a}$ (see \cite[Chapter 1]{BL}).

Notice that any eigenvalue of an endomorphism of a complex torus is necessarily an algebraic integer, since it satisfies the characteristic polynomial of the rational representation of the endomorphism.

\begin{lemma}\label{lessthan1}
If $f$ has an eigenvalue $\lambda$ with $0<|\lambda|<1$, then it has another eigenvalue with absolute value greater than 1.
\end{lemma}
\begin{proof}
Write
$$\chi_f^r(t)=\chi_f^a(t)\overline{\chi_f^a(t)}=\prod_{i=1}^g(t-\lambda_i)(t-\overline{\lambda_i}).$$
Assume first that all the eigenvalues are non-zero. Therefore since $|\chi_f^r(0)|=\prod_{i=1}^g|\lambda_i|^2$ is an integer greater than 0, we must have that there exists an eigenvalue of absolute value greater than 1. If one of the eigenvalues is equal to zero, a similar argument can be applied by dividing $\chi_f^r(t)$ by the largest power of $t$ that appears.
\end{proof}

The next proposition is proved in \cite{BH}, and says what the minimal polynomial of an algebraic integer on the unit circle must satisfy.

\begin{proposition}\label{minpol}
Let $a\in\C$ be an algebraic integer of absolute value 1 and different from $\pm1$. Then its minimal polynomial is a polynomial over $\Z$ of even degree with symmetric coefficients, whose roots occur in reciprocal pairs.
\end{proposition}
\begin{proof}
See \cite[Lemma 1.8]{BH}.
\end{proof}

In \cite{BH}, the key result that allowed for a complete classification of eigenvalues of endomorphisms on two-dimensional complex tori is the fact that if an eigenvalue is of absolute value 1, it must be a root of unity. As stated in the introduction, this is also the key problem to being able to extend the results of \cite{BH} to arbitrary dimension. We will make this precise in the following proposition. As in the introduction, write
$$F(n):=\#\mbox{Fix}(f^n).$$

\begin{proposition}
Let $X$ be a complex torus and let $f\in\mbox{End}(X)$ be an endomorphism of $X$. If every eigenvalue of $f$ that has absolute value 1 is a root of unity, then $F(n)$ exhibits exactly one of the behaviors from Theorem \ref{main}.
\end{proposition}

\begin{proof}
If $\chi_f^r(t)$ has no roots that are roots of unity, then by Lemma \ref{lessthan1} there must be an eigenvalue of absolute value greater than 1. In this case $F(n)$ clearly has exponential growth. If $\chi_f^r(t)$ only has roots that are roots of unity, then clearly $F(n)$ is periodic. Assume then that $\chi_f^r(t)$ has roots that are roots of unity, as well as roots that are not. We factor this polynomial as
$$\chi_f^r(t)=P(t)Q(t)$$
where $P(t),Q(t)\in\mathbb{Z}[t]$ and the roots of $P(t)$ consist of roots of unity and the roots of $Q(t)$ do not. We have that 
$$F(n)=\left(\prod_{|\lambda|=1}|1-\lambda^n|^2\right)\left(\prod_{|\lambda|\neq1}|1-\lambda^n|^2\right).$$
It is clear that $\prod_{|\lambda|=1}(1-\lambda^n),\prod_{|\lambda|\neq1}(1-\lambda^n)\in\mathbb{Z}$ since they are algebraic integers and the Galois group of $P(t)$ fixes the former and the Galois group of $Q(t)$ fixes the latter. Let $n_1,\ldots,n_r$ be the different orders of the roots of $Q(t)$, and set
$$h(n):=\left\{\begin{array}{ll}F(n)&\mbox{if }n\not\equiv 0\mbox{ (mod }n_1,\ldots,n_r)\\\prod_{|\lambda|\neq1}|1-\lambda^n|^2&\mbox{in any other case}\end{array}\right.$$ 
We see that $h(n)\in\mathbb{Z}_{>0}$ for all $n\in\mathbb{Z}$ and has exponential growth. This gives us growth of type (3) in Theorem \ref{main}.
\end{proof}

\begin{rem}
We observe that if $f$ is an automorphism of finite order, then its eigenvalues are roots of unity and satisfy $\chi_f^r$ which is a polynomial of degree $2g$, where $g$ is the dimension of $X$. Therefore, if $\lambda$ is an eigenvalue of order $k$ of $f$, $\varphi(k)\leq 2g$, where $\varphi$ is Euler's totient function. 
\end{rem}

In the rest of this section we will concentrate on examples of endomorphisms with eigenvalues on the unit circle but that are not roots of unity. As stated in the introduction, it seems that these endomorphisms should be of interest in their own right, and not just in the context of fixed points. For example if $\lambda$ is an eigenvalue of $f$ that lies on the unit circle and is not a root of unity and $z\in X$ is the image of an eigenvector of $\rho_a(f)$ associated to $\lambda$, then the closure of the orbit of $z$ by $f$ is the image of a closed $C^\infty$ curve $\alpha:[0,1]\to X$.

\begin{rem}
From a dynamical point of view, it is interesting when $f$ is an automorphism that does not have eigenvalues that are roots of unity since this is equivalent to $f$ being \emph{ergodic} with respect to the Haar measure on $X$.
\end{rem}

The following proposition shows that the presence of eigenvalues of absolute value 1 that are not roots of unity could be interesting in terms of dynamics.

\begin{proposition}
Let $f$ be an endomorphism of $X$ that has an eigenvalue that lies on the unit circle and is not a root of unity. Then $f$ restricts to an ergodic automorphism of a non-zero subtorus of $X$.
\end{proposition}
\begin{proof}
Let $\lambda$ be an eigenvalue of $f$ that lies on the unit circle and is not a root of unity, and let $Q(t)$ be its minimal polynomial over $\mathbb{Z}$. Then $Q(t)\mid\chi_f^r(t)$ and actually $Q(t)$ divides the minimal polynomial of $\rho_r(f)$. Let $Y$ be the subtorus $(\ker Q(f))_0$ where the 0 subscript stands for the connected component of $\ker Q(f)$ that contains 0. We see that since $f$ commutes with $Q(f)$, $f$ restricts to an endomorphism of $Y$. If $Y=0$, then $Q(f)$ is an isogeny and so $Q(\rho_r(f))$ is invertible. However this contradicts the fact that $Q(t)$ divides the minimal polynomial of $\rho_r(f)$. We see that the characteristic polynomial of $\rho_r(f|_Y)$ is a power of $Q(t)$ since the minimal polynomial of $\rho_r(f|_Y)$ (which is $Q(t)$) must have the same linear factors as the characteristic polynomial. By Proposition \ref{minpol}, 
$$Q(0)=1=\det(\rho_r(f|_Y))$$
and so $f|_Y$ is an automorphism. 
\end{proof}

We finish this section by showing that endomorphisms with eigenvalues that lie on the unit circle but are not roots of unity appear on complex tori of every dimension greater than or equal to $3$. 

\begin{proposition}
For every $g\geq3$, there exists a complex torus of dimension $g$ that has an endomorphism with an eigenvalue that lies on the unit circle but is not a root of unity.
\end{proposition}

\begin{proof} Let $E$ be the unique elliptic curve with an automorphism of order 4. We have that via the analytic representation, $\mbox{End}(E^3)\simeq M_{3\times 3}(\mathbb{Z}[\sqrt{-1}])$. Now take the endomorphism
$$f=\left(\begin{array}{ccc}0&0&-\sqrt{-1}\\1&0&-2\sqrt{-1}\\0&1&-2\end{array}\right).$$
We have that $\chi_f^r(t)=t^6+4t^5+4t^4+4t^2+4t+1$, and by the main result of \cite{KM}, this polynomial has at least two roots on the unit circle. It is easy to check that this polynomial is irreducible and is not cyclotomic, and so $f$ is the kind of endomorphism we are looking for. For general $g$, we can just take $E^3\times A$ where $A$ is a $(g-3)$-dimensional abelian variety, along with the endomorphism $f\times\mbox{id}$.

\end{proof}

In the following example, we show how many examples can be obtained from an algebraic number theory perspective. 

\begin{example}
Let $\delta\in\mathbb{C}$ be an algebraic integer that lies on the unit circle but is not a root of unity, such that $\mathbb{Q}(\delta)$ is a totally complex extension of $\mathbb{Q}$ of degree $2g$. Let $\tau_1,\ldots,\tau_{g}:L\hookrightarrow\mathbb{C}$ be the different non-pairwise conjugate embeddings of $L$ into $\mathbb{C}$. It is a well-known fact of elementary algebraic number theory that $\phi=(\tau_1,\ldots,\tau_g)$ sends the ring of algebraic integers $\mathcal{O}_L$ to a (full rank) lattice $\Lambda_L$ in $\mathbb{C}^g$. Moreover, $\mathbb{Z}[\delta]$ acts on $\Lambda_L$ diagonally by 
$$\delta:(\tau_1(x),\ldots,\tau_g(x))\mapsto(\tau_1(\delta x),\ldots,\tau_g(\delta x)).$$
This action extends to all $\mathbb{C}^g$ as a $\mathbb{C}$-linear action, and so multiplication by $\delta$ induces an endomorphism (and actually automorphism) $f$ of $\mathbb{C}^g/\Lambda_L$. Moreover, the minimal polynomial of $\delta$ over $\mathbb{Q}$ clearly divides the characteristic polynomial of the rational representation of $f$, and so $\delta$ appears as an eigenvalue of $f$. 
\end{example}

\begin{rem}
Note that such an example for the second proof exists only for $d\geq3$, since for $d=2$ the subfield $\mathbb{Q}(\delta+\overline{\delta})$ would be a degree 2 totally real extension of $\mathbb{Q}$, thus making $L$ a CM field. However, by \cite[Theorem 2]{D}, CM fields do not contain algebraic integers that lie on the unit circle and that are not roots of unity (or to be more precise, no image of an embedding of a CM field into $\mathbb{C}$ contains an element of this type).
\end{rem}

Take, for example,
$$p(t):=\sum_{k=0}^{2g}t^k-3t-3t^{2g-1}.$$
It appears that this polynomial is irreducible, is not cyclotomic and has only complex roots for all $g$ (we confirmed this by computer for $g\leq 200$). Moreover, by \cite{KM}, this polynomial has at least 2 roots on the unit circle (which are therefore not roots of unity). Hence a root of $p(t)$ that lies on the unit circle would give a $\delta$ as above.

\section{Mahler's measure and the proof of Theorem \ref{main}}

In this section we will prove that the existence of odd eigenvalues does not give us new asymptotic behavior.

If $Q(t)=a_0\prod_{i=1}^d(t-\alpha_i)$ is a non-zero polynomial in $\mathbb{C}[t]$, we define the \emph{Mahler measure} of $Q(t)$ to be
$$M(Q):=|a_0|\prod_{i=1}^d\max\{1,|\alpha_i|\}$$
and the \emph{logarithmic Mahler measure} of $Q$ to be
$$m(Q):=\log(M(Q)).$$
If $a_0=1$, then we define the quantity
$$\Delta_n(Q):=\prod_{i=1}^d(\alpha_i^n-1).$$

As before, let $f$ be an endomorphism of a complex torus $X$ with eigenvalues $\lambda_1,\ldots,\lambda_g$. The \emph{topological entropy} of $f$ is a positive number that measures in a sense the complexity of the topological dynamical system $(X,f)$. In our context for complex tori (see \cite{AW}), the topological entropy of $f$ is simply
$$h(f)=\sum_{i=1}^g\log\max\{1,|\lambda_i|\}=m(\chi_f^r).$$ 
Note moreover that $F(n)=\Delta_n(\chi_f^r)$ (no absolute value is needed since for every root of $\chi_f^r(t)$, its conjugate appears as a root as well). The lemma that follows is fundamental to our study of eigenvalues of complex tori, and essentially is a corollary of Baker's Theorem. We will sketch the proof, but for a complete proof we refer to \cite[Lemma 1.10]{EW}:

\begin{lemma}\label{mahler}
Let $Q(t)\in\mathbb{Q}[t]$ be a polynomial, and assume that none of its roots are roots of unity. Then
$$m(Q)=\lim_{n\to\infty}\frac{1}{n}\log|\Delta_n(Q)|.$$
\end{lemma}
\begin{proof}[Sketch of proof]
If $\alpha$ is a root of $Q(t)$ that does not have absolute value 1, then it is a simple exercise to prove that
$$\frac{1}{n}\log|1-\alpha^n|\to\log\max\{1,|\alpha|\}.$$ 
Assume then that $Q(t)$ has a root $\alpha$ that is of absolute value 1 and is not a root of unity. As a corollary of Baker's Theorem, by \cite[Lemma 1.11]{EW}, there exist positive constants $a,b\in\mathbb{Z}_{>0}$ such that
$$\left|\alpha^n-1\right|>\frac{a}{n^b}$$
for all $n\geq1$. Therefore
$$\log|\alpha^n-1|>\log(a)-b\log(n).$$
We will prove that every subsequence of $c_n:=\frac{1}{n}\log|\alpha^n-1|$ has a subsequence that converges to 0. If $(c_{n_k})_{k\in\mathbb{N}}$ is a subsequence, then there exists $(n_{k_\ell})_{\ell\in\mathbb{N}}\subseteq(n_k)_{k\in\mathbb{N}}$ such that $\alpha^{n_{k_\ell}}$ converges to some $z_0$ on the unit circle. If $z_0\neq1$, then clearly $c_{n_{k_\ell}}\to0$. If $z_0=1$, then for $\ell\gg0$,
$$0>\frac{1}{n_{k_\ell}}\log|\alpha^{n_{k_\ell}}-1|>\frac{1}{n_{k_\ell}}(\log(a)-b\log(n_{k_\ell}))\to0.$$
\end{proof}

Given $f$ as before, factor 
$$\chi_f^r(t)=P(t)Q(t)$$
where $P(t),Q(t)\in\mathbb{Z}[t]$ are such that the roots of $P(t)$ are all roots of unity and the roots of $Q(t)$ are not. Then
$$F(n)=\Delta_n(P)\Delta_n(Q),$$
and so in order to analyze the growth of $F(n)$ we need to understand the behavior of $\Delta_n(P)$ and $\Delta_n(Q)$. However it is clear that $\Delta_n(P)$ is periodic, and so we will analyze $\Delta_n(Q)$. The proof of the first part of Theorem \ref{main} follows immediately from the following proposition:

\begin{proposition}
The growth of $\Delta_n(Q)$ when $n\to\infty$ is exponential.
\end{proposition}
\begin{proof}
We first note that by Lemma \ref{lessthan1}, $Q(t)$ has at least one root of absolute value greater than or equal to $1$. Moreover, if every root of $Q(t)$ is of absolute value 1, then each root is a root of unity, which is a contradiction. Therefore, $Q(t)$ must have a root of absolute value greater than 1. In particular, $m(Q)>0$. Now
$$0<m(Q)=\lim_{n\to\infty}\frac{1}{n}\log|\Delta_n(Q)|$$
which means that
$$1<C:=e^{m(Q)}=\lim_{n\to\infty}|\Delta_n(Q)|^{1/n}.$$
Therefore, given $0<\epsilon<C-1$ there exists $N$ such that for all $n\geq N$, 
$$(C-\epsilon)^n<|\Delta_n(Q)|<(C+\epsilon)^n.$$
Since $C-\epsilon>1$ and $C+\epsilon>1$, this implies that the growth is exponential.
\end{proof}

\begin{example}
It is trivial that behaviors of type $(1)$ and $(2)$ appear for simple abelian varieties. Indeed, multiplication by an integer $m\in\mathbb{Z}\backslash\{\pm1\}$ gives exponential behavior on any complex torus, and multiplication by $-1$ gives periodic behavior for any complex torus. A less trivial example of behavior $(2)$ can be obtained by giving a simple abelian variety with a finite order automorphism. 
\end{example}

The proof of Theorem \ref{main} will be complete with the following:

\begin{proposition}\label{nonsimple}
Let $X$ be a complex torus and let $f\in\mbox{End}(X)$ be an endomorphism that presents behavior $(3)$. Then $X$ is not simple.
\end{proposition}
\begin{proof}
As above, write $P_f^r(t)=P(t)Q(t)$ where $P(t),Q(t)\in\mathbb{Z}[t]$ are such that the roots of $P(t)$ are all roots of unity and the roots of $Q(t)$ are not. Let $Y$ be the connected component of $\ker Q(f)$ that contains 0. If $X$ is simple, then $Y$ is either 0 or all $X$. If $Y=0$, then $Q(f)$ is an isogeny and therefore $Q(\rho_r(f))$ is invertible in $\mbox{End}(H_1(X,\mathbb{Q}))$. However, $P(t)$ and $Q(t)$ have different irreducible factors, and the irreducible factors of the minimal polynomial of $\rho_r(f)$ must be equal to the irreducible factors of $P(t)Q(t)$. Therefore if $Q(\rho_r(f))$ is invertible, $P(\rho_r(f))=0$, a contradiction. 

If $Y=X$, then $Q(\rho_r(f))=0$ and the same argument applies. Therefore $Y$ is non-trivial and $X$ is not simple.
\end{proof}

\begin{corollary}
An abelian variety $X$ is simple if and only if for every $f\in\mbox{End}(X)$, $\#\mbox{Fix}(f^n)$ is either periodic or has exponential growth.
\end{corollary}
\begin{proof}
We only need to prove that if $X$ is not simple, then there exists an endomorphism of $X$ with behavior of type $(3)$. Let $Y$ be a non-trivial abelian subvariety of $X$, and let $Z$ be its complementary abelian subvariety with respect to some polarization (this is where we need for $X$ to be an abelian variety, and not just a complex torus). If $Y\cap Z$ is contained in the group of $m$-torsion points of $X$ for some $m\geq3$, then the endomorphism
$$Y\times Z\to Y\times Z$$
$$(y,z)\mapsto(-y,(m-1)z)$$
descends to an endomorphism of $X$ which has behavior $(3)$.
\end{proof}

It is well-known that the endomorphism algebra $\mbox{End}_\mathbb{Q}(X)$ of a simple abelian variety $X$ (along with the Rosati involution), if different from $\mathbb{Q}$, takes one of four explicit forms (see \cite[Section 5.5]{BL}). In \cite{BH} each case was analyzed in order to see what behavior occured, and it was concluded that behavior $(3)$ never appears in the simple case.

Let $X$ be a simple abelian variety, and let $\mbox{End}(X)_{\text{tors}}^\times$ be the elements of finite order in $\mbox{End}(X)^\times$. Our results imply the following:

\begin{corollary}
Let $X$ be a simple abelian variety. Then for every 
$$f\in\mbox{End}(X)\backslash\mbox{End}(X)_{\text{tors}}^\times,$$ 
$\#\mbox{Fix}(f^n)$ grows exponentially.
\end{corollary}


\begin{thebibliography}{999999}

\bibitem{AW} R.L. Adler, B. Weiss. \textit{Entropy, a complete metric invariant for automorphisms of the torus}. Proc. Nat. Acad. Sci. U.S.A. 57, 1967, 1573-1576.

\bibitem{BH} T. Bauer, T. Herrig. \textit{Fixed points of endomorphisms on two-dimensional complex tori}. J. Algebra 458, 351-363 (2016).

\bibitem{BL} C. Birkenhake, H. Lange. \textit{Complex Abelian Varieties}. Grundlehren der mathematischen Wissenschaften, Volume 302, 2004. Springer Berlin Heidelberg.

\bibitem{BL2} C. Birkenhake, H. Lange. \textit{Fixed-point free automorphisms of abelian varieties}. Geom. Dedicata 51 (3) (1994) 201-213.


\bibitem{D} R. Daileda. \textit{Algebraic integers on the unit circle}. Journal of Number Theory 118 (2006), 189-191.

\bibitem{EW} G. Everest, T. Ward. \textit{Heights of Polynomials and Entropy in Algebraic Dynamics}. Universitext, 1999. Springer-Verlag London.

\bibitem{KM} J. Konvalina, V. Matache. \textit{Palindrome-polynomials with roots on the unit circle}. C. R. Math. Acad. Sci. Soc. R. Can. 26 (2004), no. 2, 39-44.

\bibitem{N} W. Narkiewicz. \textit{Elementary and Analytic Theory of Algebraic Numbers}. Springer Monographs in Mathematics, 2004. Springer-Verlag Berlin Heidelberg.

\bibitem{OT} K. Oguiso, T. Truong. \textit{Salem numbers in dynamics on K\"{a}hler threefolds and complex tori}. Math. Z. 278 (2014), no. 1-2, 93-117. 

\end{thebibliography}
\end{document}